\documentclass[11pt,a4paper,reqno]{amsart}
\usepackage{amsmath,amssymb,amsfonts,epsfig,mathrsfs,cite}
\usepackage[T1]{fontenc}
\usepackage{color}
\usepackage{array}
\usepackage{amsthm}
\usepackage{amstext}
\usepackage{graphicx}
\usepackage{setspace}
\usepackage[margin=2.5cm]{geometry}
\usepackage{color}
\usepackage{enumitem}
\usepackage{undertilde}
\usepackage{amscd,psfrag}
\usepackage{yhmath}
\usepackage[mathscr]{eucal}

\usepackage{comment}

\allowdisplaybreaks[4]

\usepackage{slashed}

\makeatletter
\pdfpageheight\paperheight
\pdfpagewidth\paperwidth

\usepackage{epstopdf}
\usepackage{chngcntr}
\counterwithin{figure}{section}
\usepackage{mathrsfs}

\usepackage{indentfirst}	

\usepackage[normalem]{ulem}
\theoremstyle{plain}
\numberwithin{equation}{section}

\newtheorem{definition}{Definition}[section]
\newtheorem{theorem}[definition]{Theorem}
\newtheorem*{theorem*}{Theorem}

\newtheorem{assumption}[definition]{Assumption}
\newtheorem{remark}[definition]{Remark}

\newtheorem*{remark*}{Remark}
\newtheorem*{sideremark*}{Side Remark}

\newtheorem*{claim*}{Claim}
\newtheorem*{q*}{Question}
\newtheorem{lemma}[definition]{Lemma}
\newtheorem{corollary}[definition]{Corollary}
\newtheorem*{corollary*}{Corollary}
\newtheorem{example}[definition]{Example}
\newtheorem{proposition}[definition]{Proposition}

\usepackage{appendix}

\newcommand{\R}{\mathbb{R}}
\newcommand{\J}{{\mathbf{J}}}
\newcommand{\KK}{{\mathcal{K}}}

\newcommand{\na}{\nabla}
\newcommand{\emb}{\hookrightarrow}

\newcommand{\id}{\text{Id}}
\newcommand{\p}{\partial}

\newcommand{\e}{\varepsilon}
\newcommand{\C}{\mathbb{C}}
\newcommand{\dd}{{\rm d}}
\newcommand{\G}{\Gamma}

\newcommand{\dvg}{\dd {\rm Vol}_g}

\newcommand{\Exp}{{\bf Exp}}

\newcommand{\tor}{{\bf T}^2}
\newcommand{\lap}{{\mathscr{L}}}

\newcommand{\diff}{{\bf Diff}}
\newcommand{\sdiff}{{\bf SDiff}}
\newcommand{\hsdiff}{{\bf H}^s \sdiff}
\newcommand{\2}{{\sqrt{-1}}}

\newcommand{\End}{{\rm End}}

\newcommand{\HH}{{\mathscr{H}}}

\newcommand{\GG}{{\mathscr{G}}}

\newcommand{\bra}{\left\langle}
\newcommand{\ket}{\right\rangle}

\newcommand{\mres}{\mathbin{\vrule height 1.6ex depth 0pt width
0.13ex\vrule height 0.13ex depth 0pt width 1.3ex}}

\newcommand{\sol}{{\mathfrak{S}}}

\newcommand{\ppl}{{\sigma_{\bf ppl}}}

\newcommand{\1}{{\mathbf{1}}}

\newcommand{\harm}{{\bf Harm}^1(\Sigma)}

\newcommand{\main}{[{\rm MAIN}]}

\newcommand{\volg}{{\rm Vol}_g}

\newcommand{\ball}{{\bf B}}

\newcommand{\rie}{{\mathcal{R}}}

\newcommand{\tpsi}{\widetilde{\psi}}

\newcommand{\tphi}{\widetilde{\phi}}

\def\Xint#1{\mathchoice
{\XXint\displaystyle\textstyle{#1}}%
{\XXint\textstyle\scriptstyle{#1}}%
{\XXint\scriptstyle\scriptscriptstyle{#1}}%
{\XXint\scriptscriptstyle\scriptscriptstyle{#1}}%
\!\int}
\def\XXint#1#2#3{{\setbox0=\hbox{$#1{#2#3}{\int}$ }
\vcenter{\hbox{$#2#3$ }}\kern-.6\wd0}}

\def\dashint{\Xint-}

\title{On analysis of the exponential map of volume-preserving diffeomorphism group on closed orientable surfaces through the vorticity}

\author{Siran Li}

\address{Siran Li: School of Mathematical Sciences $\&$ IMA-Shanghai, Shanghai Jiao Tong University, No.~6 Science Buildings,
800 Dongchuan Road, Minhang District, Shanghai, China (200240)}

\email{\texttt{siran.li@sjtu.edu.cn}}


\date{\today}

\begin{document}

\begin{abstract}
We study the exponential map of group of volume-preserving diffeomorphisms on  closed orientable surfaces via the vorticity formulation of the incompressible Euler equation. We present an alternative, fluid dynamical proof of the theorem of Ebin--Misio\l{}ek--Preston \cite{emp}: the exponential is a nonlinear Fredholm mapping of index zero. We extend Shnirelman's rigidity result \cite{s} for the exponential map from 2-dimensional flat torus to arbitrary orientable closed surfaces. That is, we prove that the exponential map is Fredholm quasiregular.
\end{abstract}
\maketitle

\section{Introduction}\label{sec: intro}

The space of volume-preserving diffeomorphisms on a  Riemannian manifold $(M,g)$ is an important topic of research in global analysis. On the one hand, it forms an infinite-dimensional Lie group that embeds into the space of $L^2$-maps from $(M,g)$ to $\R^N$, where $N$ is sufficiently large.  On the other hand, it connects naturally to PDEs in mathematical hydrodynamics: as established by Arnold in his seminal paper \cite{a} in 1966, the geodesic equation on this infinite-dimensional Lie group is precisely the Euler equation describing the motion of an  incompressible,  inviscid fluid on $(M,g)$:
\begin{equation}\label{euler}
\begin{cases}
\frac{\p v}{\p t} +v \cdot \na v + \na p = 0 \qquad\text{ in } [0,T]\times M,\\
v|_{t=0} = v_0 \qquad \text{ on } M
\end{cases}
\end{equation}
for timespan $T>0$, velocity $v \in \G(TM)$, and pressure $p: M \to \R$.

For an $n$-dimensional closed (\emph{i.e.}, compact, without boundary) Riemannian manifold $(M,g)$,  write $\diff(M)$ for the group of diffeomorphisms on $M$, and define
\begin{align*}
\sdiff(M,g) = \left\{ \phi \in \diff(M):\, \phi^\# \dvg = \dvg \right\},
\end{align*}
where $\phi^\#$ is the pullback under $\phi$ and $\dvg$ is the Riemannian volume form on $(M,g)$. We shall also consider the space $\hsdiff (M,g)$ of volume-preserving diffeomorphisms with Sobolev $H^s$-regularity. When $s > \lfloor \frac{n}{2} \rfloor + 1$, $\hsdiff(M,g)$ is an infinite-dimensional Lie group and $
\hsdiff(M,g) \emb H^s\left(M,\R^N\right)$  as a Banach submanifold.  Here $m$ is any natural number such that $(M,g)$ isometrically embeds into $\R^m$, whose existence follows from Nash embedding; see \cite{n}. Its Lie algebra is identified with the space of  incompressible vectorfields on $(M,g)$; \emph{i.e.},
\begin{equation}\label{lie alg, identification}
T_{\id} \hsdiff (M,g) = \left\{ v \in {\bf H}^s(M,TM):\, {\rm div}\, v = 0 \right\}.
\end{equation}

Here and hereafter, for any vector bundle $E$ over $M$ and $X=W^{2,p}, H^s, \ldots$, we shall use $X(M,E)$ to designate the space of $E$-sections in the regularity class $X$. One also write $\G(E)$ for the space of $E$-sections (with unspecified sufficient regularity; usually assumed to be $C^\infty$ in this paper). The divergence operator ${\rm div}: {\bf H}^s(M,TM) \to {\bf H}^{s-1}(M,\R)$ is defined as usual via duality. See Appendix \ref{sec: appendix}.

In 1970, Ebin--Marsden \cite{em} proved that the infinite-dimensional manifold $\hsdiff(M,g)$ admits a smooth Levi-Civita connection and that,  given an initial velocity in $T_{\bullet} \hsdiff(M,g)$, the geodesic equation is locally well-posed. Thus, there is a well defined exponential map
\begin{align}\label{exp}
\Exp: \mathcal{N} \subset T_{\id}\hsdiff(M,g) \, \longrightarrow  \,  \hsdiff(M,g)
\end{align}
where $\mathcal{N}$ is a neighbourhood of $0$. It is classically known that $\Exp$ is smooth and invertible  in the vicinity of $0$. Recall \eqref{lie alg, identification} for a characterisation of $T_{\id}\hsdiff(M,g)$.

From now on, we denote by $(\Sigma,g)$ a closed orientable surface.  For $M=\Sigma$, the domain $\mathcal{N}$ of $\Exp$  in \eqref{exp} can be taken as the whole Lie algebra $T_{\id}\hsdiff(\Sigma,g)$. The global behaviour of $\Exp$ on $T_{\id}\hsdiff(\Sigma,g)$, however, has enormous subtleties. Misio\l{}ek \cite{m} proved the existence of conjugate points along geodesics on $\hsdiff(\Sigma=\tor)$, where $\tor$ = flat torus. In particular, there exists $v \in T_{\id}\hsdiff(\tor)$ with $\dim\ker(\dd_v\Exp) \geq 1$ (``mono-conjugate'') for $$\dd_v \Exp: T_vT_{\id}\hsdiff(\tor) \longrightarrow T_{\Exp(v)}\hsdiff(\tor).$$ See also Shnirelman \cite{s1} for  analogous result on 3-ball.

On the other hand, for the volume-preserving diffeomorphism group on 2D closed surfaces, $\Exp$ is just a little worse than being conjugate point-free: as shown by Ebin--Misio\l{}ek--Preston \cite[Theorem~1]{emp}, $\Exp:T_{\id} \hsdiff(\Sigma,g) \to \hsdiff(\Sigma,g)$ is \emph{Fredholm of index zero}.   

\begin{definition}
Let $X$ and $Y$ be Banach spaces. A smooth, nonlinear map $f:X \to Y$ is {\bf Fredholm} if the Fr\'{e}chet derivative $\dd_pf: X \to Y$ is a Fredholm operator at each $p \in X$. The {\bf Fredholm index} of $f$, denoted as ${\rm Ind} (f)$, is the index of $\dd_pf =\dd f\big|_p$. 
\end{definition}
Here, recall that
\begin{align*}
{\rm Ind} (f) := {\rm Ind } \left(\dd_p f\right) = \dim \ker \left(\dd_p f\right) - \dim {\rm coker} \left(\dd_p f\right).
\end{align*}
It is independent of $p$ when the domain of $f$ is connected; \emph{e.g.}, in our case the domain is a Banach space. The study of Fredholmness of \emph{nonlinear} maps
is pioneered by Smale \cite{smale}.

A refined characterisation of $\Exp$ in the case $\Sigma = \tor$ was obtained in \cite{s} by Shnirelman, who proved that $\Exp$ possesses certain global rigidity in its geometrical structures. More precisely, the nonlinear functional-analytic notions below  were introduced by Shnirelman in 1970s: 
\begin{definition}[Shnirelman \cite{s0}]
Let $X$ and $Y$ be Banach spaces, and let $F: X \to Y$ be a  continuous map that is nonlinear in general. Then 
\begin{enumerate}
\item
$F$ is said to be a {\bf ruled map} if $X$ is foliated into affine subspaces $X^\alpha \subset X$ indexed by $\R^k$, \emph{i.e.}, $X = \bigsqcup_{\alpha \in \R^k} X^\alpha$, such that $F|_{X^\alpha}: X^\alpha \to Y$ is affine for each $\alpha$, and that $F|_{X^\alpha}$ depends continuously on $\alpha$.
\item
$F$ is {\bf quasiruled} if it can be approximated locally uniformly by ruled maps $\{F_k\}$.
\item
$F$ is {\bf Fredholm quasiruled} if $F$ is quasiruled and the approximating ruled maps $\{F_k\}$ as in (2) above satisfy the following conditions: $Y^\alpha := {\rm range}\, \left(F_k|_{X^\alpha}\right)$ is closed for each $\alpha$, ${\rm codim}\,Y^\alpha = {\rm codim}\,X^\alpha = k$, $F_k|_{X^\alpha}: X^\alpha \to Y^\alpha$ are bijective, and $\left(F_k|_{X^\alpha}\right)^{-1}$ are locally uniformly continuous for all sufficiently large $k$. (Here $k$ and $X^\alpha$ are as in (1) above.)
\item
Assume that $X^-$, $Y^{-}$ are dense, compactly embedded subsets of $X$ and $Y$, respectively. A continuous map $f: X^- \to Y^-$ is said to be {\bf quasilinear} if (i), it can be extended to a continuous map $F: X \to Y$, and (ii), there exist continuous  affine maps $\{B_u : X^- \to Y^-\}_{u \in X}$ which depends continuously on the parameter $u$, such that $A(u) =B_u (u)$ for each $u \in X^-$.
\end{enumerate}
\end{definition}

Fredholm quasiruled maps have nice topological--geometrical properties. For instance, they admit a definition of topological degree, thus leading to no-wandering properties (preservation of domain). Moreover, a notion of \emph{Fredholm quasiruled Banach manifolds} can be naturally defined; it forms a category, for which Fredholm quasiruled maps are  morphisms. As a primary example, $\hsdiff(M,g)$ for $s > \lfloor\frac{n}{2}\rfloor + 1$ is a Fredholm quasiruled Banach manifold. See \cite{s0} for details.

In \cite{s}, by using tools from ``microglobal'' analysis (coined word: microlocal + global analysis), Shnirelman proved that $\Exp$ is Fredholm quasiruled when $\Sigma = \tor$, hence the geodesic flow on $\hsdiff(\tor)$ for $s>2$ is a Fredholm quasiruled flow.

To summarise, the following results  have been established concerning the exponential map on  volume-preserving diffeomorphism group of surfaces. See Ebin--Misio\l{}ek--Preston \cite[Theorem~1]{emp} $\&$ Shnirelman \cite[Theorem~3.1]{s}:
\begin{theorem}\label{thm, earlier}
Let $(\Sigma,g)$ be a compact surface. Then $\Exp:T_{\id} \hsdiff(\Sigma,g) \to \hsdiff(\Sigma,g)$ is  a nonlinear Fredholm map of index zero as long as $s>2$. Moreover, for $(\Sigma,g)=\tor=$  the flat torus, $\Exp$ is Fredholm quasiruled. 
\end{theorem}

At the end of the paper, Shnirelman \cite[p.S395]{s} remarked --- 
\begin{quotation}
``Note that we have used the global flat structure on the torus. The proof of the analogous theorem for the fluid motion on a curved surface should require additional devices.''
\end{quotation}

The goal of this note is to address this problem. We establish the following.
\begin{theorem}\label{thm}
Theorem~\ref{thm, earlier} is valid for any closed orientable surface $(\Sigma,g)$. That is, $\Exp:T_{\id} \hsdiff(\Sigma,g) \to \hsdiff(\Sigma,g)$ is  a Fredholm quasiruled map of index zero whenever  $s>2$.
\end{theorem}

The main difficulty for proving Theorem~\ref{thm}, \emph{i.e.}, for extending Theorem~\ref{thm, earlier} to general compact orientable surface $\Sigma$ other than $\tor$, lies in the lack of Fourier analytic tools on $\Sigma$. To our remedy, however,  paradifferential calculus has been  developed recently on complete Riemannian manifolds with doubling property, lower volume bound for unit balls, and  Poincar\'{e} inequality. See Bernicot \cite{ber} and Bernicot--Sire \cite{bs}. We will collect some of these tools in \S\ref{section: paradiff} below.

Next, as vorticity is being transported by the Euler flow on $(\Sigma,g)$, it is handy to express $\Exp(v)$ for $v\in T_{\id} \hsdiff(\Sigma,g)$ as a function of the vorticity $\omega = \dd v \in {\bf H}^{s-1}\left(\Sigma, \bigwedge^2T^*\Sigma \right)$ which, by duality, can be identified with a scalarfield. The recovery of $v$ from $\omega$ is achieved via a pseudodifferential operator of order $-1$, known as the \emph{Biot--Savart operator} in the Euclidean case. We shall develop its analogue on $(\Sigma,g)$ in \S\ref{section: BS}. This will be done via the Green's operator for the Laplace--Beltrami operator $\Delta_g: \Omega^1(\Sigma) \to \Omega^2(\Sigma)$ on differential $1$-forms. See the classical work \cite{dr} by de Rham.

The tools elaborated in \S\S\ref{section: paradiff} $\&$ \ref{section: BS} enable us to prove  Theorem~\ref{thm} in \S\ref{section: proof}. We shall essentially follow Shnirelman's work \cite{s} for the  case $\Sigma = \tor$;  nevertheless, various new arguments are presented to deal with the non-flat geometry.

Apart from the proof for the rigidity (\emph{i.e.}, Fredholm quasiregular) property of $\Exp$ on closed orientable surfaces, this paper also provides an alternative proof  for (the orientable case of) Ebin--Misio\l{}ek--Preston \cite[Theorem~1]{emp}, which states that $\Exp$ on closed 2-dimensional surfaces is Fredholm of index zero. Our new proof is fluid dynamical in nature, which exploits the vorticity formulation of the 2-dimensional incompressible Euler equation. For future investigations, we hope to generalise our approach to non-orientable closed surfaces and to surfaces-with-boundary.

\smallskip
\noindent
{\bf Notations.} Throughout the paper, $(\Sigma,g)$ is a closed orientable Riemannian manifold of dimension 2. The superscript/subscript hash $\#$ always denotes the pullback/pushforward, while the superscript asterisk $*$ is reserved for adjoint operator. For a constant $C$, by writing $C=C(M,g)$ we mean that $C$ depends only on the geometry of $(M,g)$. An inner product is denoted as $\bullet$ if it is taken with respect to a Euclidean metric. We shall write $\na$ for Riemannian gradient and $D$ for Euclidean gradient. The index $s$ is always a number strictly greater than $2 = \lfloor\frac{\dim \Sigma}{2}\rfloor + 1$, hence ${\bf H}^s \emb L^\infty$ over $\Sigma$. The index $\sigma$ is an arbitrary number strictly greater than $s$.

\section{Paradifferential calculus on $(\Sigma,g)$}\label{section: paradiff}

In this work, we make essential use of the theory of paradifferential calculus for Riemannian manifolds with sub-Laplacian structures (for which compact surface $(\Sigma,g)$ is a special case) developed by Bernicot--Sire \cite{bs}.

\subsection{Geometric assumptions}

Let $(M,g)$ be a Riemannian manifold and let $\lap$ be a sub-Laplacian on $(M,g)$. That is, $\lap  = -\sum_{k=1}^K X_k^2$ where $\{X_k\}$ are real-valued vectorfields.  \begin{assumption}\label{assumption}
Suppose that $(M,g,\lap)$ satisfy the following (see \cite[p.941]{bs}):
\begin{enumerate}
\item
$(M,g)$ is (uniform) doubling, \emph{i.e.}, there exists a uniform constant $C_1>0$ such that 
\begin{align*}
\volg(\ball(x,2r)) \leq C_1 \volg(\ball(x,r))
\end{align*} 
for any $x \in M$ and $r>0$;
\item
$(M,g)$ supports the 1-Poincar\'{e} inequality, namely that there exists a uniform constant $C_2>0$ such that for each $f \in C^\infty_0(M)$, $r>0$, and any $Q$ = geodesic ball of radius $r$, 
\begin{align*}
\dashint_Q |f-f_Q|\,\dvg \leq C_2\, r\,  \dashint_Q |\na f|\,\dvg.
\end{align*}
\item
There exists a uniform constant $c_3>0$ such that $\volg(\ball(x,1)) \geq c_3$ for all $x \in M$;
\item
All the local Riesz transforms $\rie_I$ and $\overline{\rie_I}$ for arbitrary multiindex $I$ are $L^p \to L^p$-bounded for each $p \in ]1,\infty[$;
\item
The sub-Laplacian $\lap$ satisfies \cite[Assumption~1.11]{bs}: $\lap$ is injective and of type $\omega$ for some $\omega \in [0,\pi/2[$ on $L^2$, and there exists $\delta>1$ such that 
\begin{enumerate}
\item
For every $z \in S_{\pi/2 - \omega}$, the linear operator $e^{-z\lap}$ is given by a kernel $p_z$ such that
\begin{align*}
\left|p_z(x,y)\right| \leq \frac{C_4}{\volg(\ball(x, \sqrt{|z|}))} \left(1+\frac{d_g(x,y)}{\sqrt{|z|}}\right)^{-\log_2(C_1)-2N-\delta};
\end{align*}
\item
$\lap$ has a bounded $H_\infty$-calculus on $L^2$;
\item
The Riesz transform $\rie:=\na\lap^{-1/2}$ is $L^p\to L^p$-bounded for each $p \in ]1,\infty[$. 
\end{enumerate}
\end{enumerate}
\end{assumption}

Here and hereafter, $\na$ denotes the Riemannian gradient, $\ball(x,r)$ denotes the geodesic balls, $d_g$ is the Riemannian distance, $|\bullet|$ is the Riemannian length of vectors, and $\volg$ is the Riemannian volume on $(M,g)$. All the $L^p$-spaces and averaged integrals $\dashint$ are taken with respect to the Riemannian volume form $\dvg$. The number $N \geq 0$ in (5a) is the uniform constant in
\begin{align}\label{N}
\volg(\ball(y,r)) \leq C_5 \left(1+\frac{d_g(x,y)}{r}\right)^N \volg(\ball(x,r)) \qquad \text{for all $x,y \in M$ and $r>0$},
\end{align}
where $C_5$ is another uniform constant. By $\lap$ being of type $\omega$ on $L^2$ we mean that $\lap$ is a closed operator with spectrum in the sector $S_\omega:= \left\{z \in \C:\,|{\rm arg}(z)|\leq \omega\right\} \cup \{0\}$, such that for each $\nu > \omega$ there is a constant $c_\nu$ such that 
\begin{align*}
\left\|(\lap - \lambda\1)^{-1}\right\|_{L^2 \to L^2} \leq \frac{c_\nu}{|\lambda|}\qquad \text{ for all } \lambda \notin S_\nu.
\end{align*}
Also, that $\lap$ has a bounded $H_\infty$-calculus on $L^2$ means that there exists $c_\nu'$ such that for each $b \in H_\infty({\rm int}\,S_\nu)$ with $\nu > \omega$, there holds
\begin{align*}
\left\|b_\lap\right\|_{L^2 \to L^2} \leq c_{\nu}' \|b\|_{L^\infty}.
\end{align*}
Here $H_\infty$ denotes the space of bounded holomorphic functions. Moreover, for each multiindex $I$, set $X_I:=\prod_{i \in I}X_i$; then the local Riesz transforms are defined as 
\begin{align}\label{riesz}
\rie_I := X_I \left(\1+\lap\right)^{-|I|/2},\qquad \overline{\rie_I} := \left(\1+\lap\right)^{-|I|/2}X_I.
\end{align}

The constants $C_1$, $C_2$, $c_3$, $C_4$, $C_5$, and $N$ all depend only on the geometry of $(M,g)$. The constants $c_\nu$ and $c_\nu'$ depend in addition on $\nu$.

Note that any closed manifold $(M,g)$ satisfies Assumption~\ref{assumption}. Indeed, it is known that a Riemannian manifold $(M,g)$ with bounded geometry --- that is, $\left\|\na^k{\rm Riem}_g\right\|_{L^\infty} \leq C(k,M,g)<\infty$ for each $k \in \mathbb{N}$, ${\rm Ric}_g > -\infty$, and the injectivity radius is positive --- satisfies (2) and (5). In fact, the Poincar\'{e} inequality holds for any $p \in [1,\infty[$, and (2) is the case $p=1$. See Coulhon--Russ--Tardivel-Nachef \cite{crt}, Triebel \cite{t}, Bernicot--Sire \cite[\S 2.1.4]{bs}. In addition, (1), (3), and the implication (5c) $ \Rightarrow$ (4) follow immediately from compactness.  

\subsection{Paraproduct}

Let $c_0= c_0(M,g)$ be a suitable uniform constant and consider 
\begin{equation*}
\psi(x):=c_0 x^Ne^{-x}(1-e^{-x}), \qquad \phi(x) = -\int_x^\infty \psi(y)\,\dd y
\end{equation*}
as well as
\begin{equation*}
\tpsi(x) = \frac{\psi(x)}{x},\qquad \tphi(x):= -\int_x^\infty \tpsi(y)\,\dd y, \qquad \psi_t(\bullet):= \psi(t\bullet),\qquad \phi_t(\bullet):= \phi(t\bullet).
\end{equation*}
Under Assumption~\ref{assumption}, the construction of \emph{paraproduct} can be generalised to $(M,g,\lap)$. Paradifferential calculus has been an important tool for the analysis of propagation of singularities in nonlinear PDEs over $\R^n$, pioneered by Bony \cite{bony} and Alinhac \cite{a}.

Our presentation of the  theorem below, which is essentially a generalisation of Bony's decomposition, follows Bernicot--Sire \cite[Definition~4.1, Theorem~4.5, and Corollary~5.2]{bs}. 
\begin{theorem} \label{thm: para}
Let $(M,g,\lap)$ be a Riemannian manifold with sub-Laplacian satisfying Assumption~\ref{assumption}. For $f,h: C^\infty_0(M)\to\R$, define the paraproduct of $f$ by $h$ by
\begin{equation}\label{paraprod}
\Pi_h(f) := -\int_0^\infty \tpsi(t\lap) \left[t\lap\tphi(t\lap)f\tphi(t\lap)h \right]\frac{\dd t}{t} - \int_0^\infty \tphi(t\lap)\left[\psi(t\lap)f \tphi(t\lap)h\right]\frac{\dd t}{t}.
\end{equation}
\begin{enumerate}
\item
For $p,q \in ]1,\infty]$  with $0<r^{-1} = p^{-1} + q^{-1}$,  $N > \frac{\log_2 (C_1)}{2}$, and $s \in ]0,2N-4[$, the mapping $$(f,h) \longmapsto \Pi_hf$$ is bounded from $W^{s,p} \times L^q$ into $W^{s,r}$. 
\item
If, in addition, $p \in ]1,\infty[$, $\e>0$, $s>\frac{\log_2(C_1)}{p}$, and $s<\frac{N-2}{2}$, the mapping
\begin{align*}
{\bf Rest}(f,h):= fh- \Pi_fh - \Pi_hf 
\end{align*}
satisfies
\begin{align*}
\left\|{\bf Rest}(f,h)\right\|_{W^{2s-\frac{d}{p},p}} \lesssim \|f\|_{W^{s+\e,p}}\|h\|_{W^{s+\e,p}}.
\end{align*}
\end{enumerate}
\end{theorem}

\begin{remark}
The parameter $N \in \mathbb{N}$ is taken such that $N \gg 1$. This would validate the choice of $s \in \left]0, \min\left\{2N-4,\frac{N-2}{2}\right\}\right[$, for which both statements in Theorem~\ref{thm: para} hold, as well as the requirement that $N > \frac{\log_2 (C_1)}{2}$. We are uncertain if in \cite{bs} this parameter must be the same $N$ as in \eqref{N} and Assumption~\ref{assumption}, (5a). But, for our purpose, on a compact manifold $(M,g)$ we can take $N \gg 1$ by enlarging $C_4(M,g)$ and $C_5(M,g)$ in \eqref{N} and Assumption~\ref{assumption}, (5a).

The number $\log_2(C_1)$ is the homogeneous dimension of $(M,g)$; here $C_1$ is the (uniform) doubling constant in Assumption~\ref{assumption} (1).
\end{remark}


On a closed manifold $(M,g)$ we may construct the paraproduct via the arguments in Shnirelman \cite[Equation(2.2)]{s}, which are taken from \cite{s0} by the same author:
\begin{align*}
f h &\approx D^{-k}\na^k (fh)\\
&=\underbrace{ D^{-k}(h \na^kf)}_{=:\,\Pi_hf} + \underbrace{D^{-k}(f\na^kh)}_{=:\,\Pi_fh} + \underbrace{D^{-k}(k\na h\cdot \na^{k-1}f + \ldots + k\na f \cdot \na^{k-1}h)}_{=:\,{\bf Rest}(f,h)}.
\end{align*} 
Here $D^{-k}$ can be rigorously defined via (local) Riesz transforms. In \cite{s} $\Pi_hf$ is denoted as $T_hf$; here we choose the former notation to avoid confusion with the tangent bundle $T\Sigma$.

For any $k \in \mathbb{N}$ let us consider $X_I = \prod_{i \in I}X_i$, where the multiindex $I$ has valency $|I|=k$. Recall the sub-Laplacian $\lap = -\sum_{i=1}^K X_i^2$. We obtain by Assumption~\ref{assumption} (4), the definition of local Riesz transform in \eqref{riesz}, and the Leibniz rule, that
\begin{align}\label{paraproduct}
fh&\approx \overline{\rie_I} (fh)\nonumber\\
&= \left(\1 + \lap\right)^{-\frac{|I|}{2}} X_I(fh)\nonumber\\
&=\underbrace{ \left(\1 + \lap\right)^{-\frac{|I|}{2}}(h X_If)}_{=:\,\Pi_hf} + \underbrace{\left(\1 + \lap\right)^{-\frac{|I|}{2}}(f X_I h)}_{=:\,\Pi_fh} \nonumber\\
&\qquad + \underbrace{\left(\1 + \lap\right)^{-\frac{|I|}{2}}\left( \sum_{J+K=I,\,1 \leq |J| \leq |I|-1} X_Jf\,X_Kh \right) }_{=:\,{\bf Rest}(f,h)}.
\end{align}
On closed manifold $(M,g)$ we may take $\lap = -\Delta_g$. This is the Hodge Laplacian with $\lap = \sum_{i=1}^{\dim M}\na_{i}\na_{i}$ in local co-ordinates. For $h \in L^\infty(M;\R)$ and $f \in {\bf H}^s(M;\R)$, the first line means that $fh - \overline{\mathcal{R}_I } (fh)\in {\bf H}^\sigma(M;\R)$ for some $\sigma > s$.

\subsection{Paracomposition}
As another tool for paradifferential calculus, we can define the \emph{paracomposition} on $(M,g,\lap)$ based on the paraproduct. For $\varpi \in  \diff(M)$ and $f:M\to\R$, consider
\begin{align}\label{composition}
f\circ\varpi &\approx \overline{\rie_I} (f \circ \varpi) \nonumber\\
&= \left(\1 + \lap\right)^{-\frac{|I|}{2}} X_I(f\circ\varpi).
\end{align}
To be explicit, we first compute for a single vectorfield $X\in\G(TM)$ that 
\begin{align*}
X(f\circ\varpi) &= \dd(f\circ\varpi)(X)\\ &= \left(\dd f\circ\varpi\right) \left[\dd\varpi(X)\right]= \bra \dd\omega(X), \na f\circ\omega \ket,
\end{align*}
where $\bra\bullet,\bullet\ket = g(\bullet,\bullet)$. Taking one further derivative with respect to $Y \in \G(TM)$, we have
\begin{align*}
YX(f \circ \varpi) &= Y\bra \dd\varpi(X), \na f\circ\varpi \ket\\
&= \bra \left(\na_Y \dd\varpi\right)(X) + \dd\varpi\left(\na_YX\right),\,\na f\circ\varpi\ket + \left[\na\na f\circ\varpi\right]\left(\dd\varpi(X),\,\dd\varpi(Y)\right).
\end{align*}
Here $\na\na f \in \G(T^*M \otimes T^*M)$ is given by $\na\na f(Z,W):=\na_Z\na_W f$. A simple induction shows that for any multiindex $I=(i_1, \ldots, i_k)$ one has
\begin{align}\label{composition2}
X_I(f\circ\varpi) &= \bra X_{i_k}X_{i_{k-1}} \cdots X_{i_2}\dd\varpi(X_{i_1}),\na f\circ\varpi\ket \nonumber\\
&\quad + \left[\left(\overbrace{\na \circ \cdots\circ \na}^{k \,\text{times}} f\right) \circ \varpi \right] \left( \dd\varpi(X_{i_1}),\,\dd\varpi(X_{i_2}),\,\cdots,\,\dd\varpi(X_{i_k})\right) + \overline{\rm Rem}.
\end{align}
We denote by $\overbrace{\na \circ \cdots\circ \na}^{k \,\text{times}} f =: \na^kf \in \G\left({T^*M}^{\otimes\,k}\right)$ the map $\na^kf(Z_1, \ldots, Z_k) := \na_{Z_k}\na_{Z_{k-1}} \cdots \na_{Z_1} f$ for any $Z_1, \ldots, Z_k \in \G(TM)$. The remainder $\overline{\rm Rem}$ consists of those terms containing no more than $(k-1)$ derivatives on both $\varpi$ and $f$; \emph{i.e.}, those involving $\na_{X_k}X_\ell$ for some $X_k, X_\ell \in X_I$. 

\begin{definition}\label{def: paracomp} Let $(M,g,\lap)$ be a Riemannian manifold with sub-Laplacian satisfying Assumption~\ref{assumption}, $\varpi \in \diff(M)$, and $f:M\to\R$. The   paracomposition of $f$ with $\varpi$ is defined as 
\begin{align*}
\KK_\varpi f := (\1+\lap)^{-\frac{|I|}{2}} \left[\na^k f \circ \varpi\right]\left( \dd\varpi(X_{i_1}),\,\dd\varpi(X_{i_2}),\,\cdots,\,\dd\varpi(X_{i_k})\right). 
\end{align*}
\end{definition}

Note that the paracomposition $\KK_\varpi f$ 
carries only the singularities of $f$.

By \eqref{composition} and \eqref{composition2} we have
\begin{align}\label{paracomposition}
f \circ \varpi &\approx \KK_\varpi f + (\1+\lap)^{-\frac{|I|}{2}}\bra X_{i_k}X_{i_{k-1}} \cdots X_{i_2}\dd\varpi(X_{i_1}),\na f\circ\varpi\ket \nonumber\\
&\qquad + (\1+\lap)^{-\frac{|I|}{2}}\left\{\overline{\rm Rem}\right\}.
\end{align}
The second term on the right-hand side of \eqref{paracomposition} can be simplified as follows. Observe that the diffeomorphism $\varpi \in C^\infty\diff(M)$ can be regarded as a map from $M$ into $\R^m$, if $(M,g)$ embeds isometrically into $\R^m$ by Nash's embedding \cite{n}. Then, view $\varpi = \left(\varpi^1, \ldots, \varpi^m\right)^\top$ and identify $f$ with an arbitrary Lipschitz extension $f: \R^m \to \R$, without relabelling. In this way
\begin{align*}
\bra X_{i_k}X_{i_{k-1}} \cdots X_{i_2}\dd\varpi(X_{i_1}),\na f\circ\varpi\ket &= \sum_{j=1}^m \left(D_jf  \circ \varpi\right)\left(X_{i_k}X_{i_{k-1}} \cdots X_{i_2}X_{i_1}\varpi^j\right)\\
&\equiv \sum_{j=1}^m \left(D_jf  \circ \varpi\right) \left(X_I\varpi^j\right),
\end{align*} 
where $D$ is the Euclidean gradient on $\R^m$. Taking $\left(\1 + \lap\right)^{-\frac{|I|}{2}}$ on both sides and recalling the paraproduct \eqref{paraproduct}, we can rephrase \eqref{paracomposition}
 as
 \begin{equation}\label{paracomposition, simplified}
 f\circ \varpi \approx \KK_\varpi f + \sum_{j=1}^m \Pi_{D_j f\circ\varpi} \left(\varpi^j\right) + (\1+\lap)^{-\frac{|I|}{2}}\left\{\overline{\rm Rem}\right\},
 \end{equation}
with $\overline{\rm Rem}$ containing $\leq (k-1)$ derivatives on $\varpi$ and $f$.

For our applications in the subsequent developments, we shall apply the  construction and  corresponding estimates for the paracomposition in the following form.

\begin{corollary}\label{cor: paracomposition}
Let $(\Sigma,g)$ be a smooth, oriented, closed surface smoothly isometrically embedded in $\R^m$. Let $\varpi \in \hsdiff(\Sigma)$ and $f \in \mathbf{H}^\ell(\Sigma;\R)$ with $s>2$ and $\ell>1$. Then
\begin{align*}
 f\circ \varpi - \KK_\varpi f - \sum_{j=1}^m \Pi_{D_j f\circ\varpi} \left(\varpi^j\right) \in {\bf H}^\varrho(\Sigma) \qquad \text{for some } \varrho > \ell.
\end{align*}
\end{corollary}
\begin{proof}
It follows directly from the $H^\varrho \to H^\varrho$ boundedness of  local Riesz transform in \eqref{composition}, the expression~\eqref{paracomposition, simplified}, Definition~\ref{def: paracomp} of paracomposition, and  Theorem~\ref{thm: para} for paraproduct.  \end{proof}

\begin{remark}\label{rem: isom emb and para}
In the case that $(M,g)$ is a closed Riemannian manifold, our constructions above for paraproduct and paracomposition are intrinsic. Thus, if $(M,g)$ is smoothly isometrically embedded into some Euclidean domain $\R^m$, then they agree with usual paradifferential calculus constructions on $\R^m$, with suitable restrictions and pullbacks to (the embedded image of) $\Sigma$. 
\end{remark}

\section{Biot--Savart operator}\label{section: BS}

Let $v \in \G(T\Sigma)$ be a divergence-free vectorfield. Define its vorticity as
\begin{equation}\label{vorticity}
\omega := \star \left[\dd\left(v^\sharp\right)\right],
\end{equation} 
with the Hodge star $\star: \Omega^2(\Sigma) \to \Omega^0(\Sigma)$. See Appendix \ref{sec: appendix} for details. We shall also write \eqref{vorticity} more compactly as $\omega = \star \dd v^\sharp$, and shall refer to the mapping $v \mapsto \omega$ as the curl or rot operator.

\subsection{Vorticity is transported}

As with the Euclidean flat case, the vorticity will never aggregate or deplete for any time:
\begin{lemma}\label{lem: vorticity}
Assume that $v \in \G(T\Sigma)$ satisfies the Euler equation. Its vorticity $\omega$ is transported along the Lagrangian trajectories. That is,
\begin{align*}
\p_t \omega + \na_v \omega = 0\qquad \text{ on } [0,\infty[ \times \Sigma.
\end{align*}
\end{lemma}

\begin{proof}
We make use of a nice expression for the curl/rot operator on surfaces; see also Samavaki--Tuomela \cite[p.11 and Appendix B]{st}. Let $\J \in \G(T^*\Sigma \otimes T\Sigma)$ be the almost complex structure arising from the Riemannian metric $g$. In local co-ordinates, $\J = \dd x^1 \otimes e_2 - \dd x^2 \otimes e_1$. Then for each vectorfield $u \in \G(T\Sigma)$ one has 
\begin{align*}
\star \left[ \dd\left( u^\sharp\right)\right] = \J^\alpha_\beta \na_\alpha u^\beta,
\end{align*}
where $\na$ is the Levi-Civita connection on $\Sigma$. Note also that $\na \J =0$.

Let us take $\sharp$, $\dd$, and $\star$ to the Euler equation $$\p_t v +\na_v v + \na p=0$$ in this given order. As $\na$ is dual to $\dd$ and $\dd^2=0$, the pressure term is eliminated. We obtain from the definition of vorticity in \eqref{vorticity} that
\begin{align}\label{x1}
\p_t \omega + \J^\alpha_\beta \na_\alpha \left(\na_v v\right)^\beta = 0.
\end{align}

To proceed, we first note that $$\left(\na_v v\right)^\beta = \na_v \left(v^\beta\right),$$ as both sides are equal to $v^\gamma \p_\gamma v^\beta + v^\gamma v^\delta \G^\beta_{\gamma\delta}$, where $\G^\beta_{\gamma\delta}$ are the Christoffel symbols for $\na$. Then 
\begin{align*}
\J^\alpha_\beta \na_\alpha \left(\na_v v\right)^\beta &= \J^\alpha_\beta \left(\na_\alpha v^\gamma\right)\left(\na_\gamma v^\beta\right) +   \J^\alpha_\beta v^\gamma \na_\alpha \na_\gamma v^\beta\\
&= \J^\alpha_\beta \left(\na_\alpha v^\gamma\right)\left(\na_\gamma v^\beta\right) +   \J^\alpha_\beta v^\gamma \na_\gamma\na_\alpha  v^\beta - \J^\alpha_\beta v^\gamma v^\delta {\rm Riem}^\beta_{\gamma\alpha\delta},
\end{align*}
thanks to the definition of Riemann curvature tensor ${\rm Riem}^\beta_{\gamma\alpha\delta}$. On 2D surface the last term vanishes, as $\J^\alpha_\beta v^\gamma v^\delta {\rm Riem}^\beta_{\gamma\alpha\delta} = K\J_{\gamma\delta} v^\gamma v^\delta = 0$. Here $K$ is the Gauss curvature of $(\Sigma,g)$, and the above identity holds by the antisymmetry of $\J$. The second term satisfies
\begin{align*}
  \J^\alpha_\beta v^\gamma \na_\gamma\na_\alpha  v^\beta = v^\gamma \na_\gamma \left(\J^\alpha_\beta \na_\alpha v^\beta\right) =  v^\gamma \na_\gamma \omega = \na_v\omega.
\end{align*} 
Finally, for the first term on the right-hand side, we compute directly to get
\begin{align*}
\J^\alpha_\beta \left(\na_\alpha v^\gamma\right)\left(\na_\gamma v^\beta\right) &= \left(\na_1 v^1\right) \left(\na_1 v^2\right) +  \left(\na_1 v^2\right) \left(\na_2v^2\right) -  \left(\na_2 v^1\right) \left(\na_1 v^1\right) -  \left(\na_2 v^2\right) \left(\na_2 v^1\right)\\
&=\left[ \left(\na_1 v^1\right)+ \left(\na_2 v^2\right)\right]\left[ \left(\na_1 v^2\right)- \left(\na_2 v^1\right)\right]\\
&\equiv \left({\rm div}\, v \right) \omega.
\end{align*}

In this way, \eqref{x1} becomes $\p_t\omega + \left({\rm div}\, v \right) \omega + \na_v\omega = 0.$ So the proof is  complete by the divergence-free condition for $v$.  \end{proof}

\subsection{The Biot--Savart operator}

Now let us discuss how to invert the curl/rot operator on closed orientable surface $(\Sigma, g)$. That is, we look for an operator 
\begin{equation*}
\sol: C^\infty(\Sigma,\R) = \Omega^0(\G) \,\longrightarrow \, \G(T\Sigma)
\end{equation*}
defined by
\begin{equation}\label{biot-savart}
\sol(\omega) = v \qquad \text{such that } \star \dd v^\sharp = \omega \text{ and } {\rm div}\,v = 0.
\end{equation}
The solution for $v$ is in general non-unique. To make the solution operator $\sol$ well defined, for initial data $v_0 \in \G(T\Sigma)$ we further require that
\begin{equation}\label{biot-savart, cohomol}
\text{$v$ is cohomologous to $v_0$ where ${\rm div}\,v_0=0$.}
\end{equation}
We call $\sol$ specified by \eqref{biot-savart} $\&$ \eqref{biot-savart, cohomol} the Biot--Savart operator on $(\Sigma,g)$. It is also widely denoted as $\sol = {\rm rot}^{-1}$ in the literature; see \cite{s} and the references cited therein. 

To determine $\sol$, we take $\star\dd$ to the vorticity in \eqref{vorticity} to obtain
\begin{align*}
\star \dd \omega = \star \dd \star \dd v^\sharp = \dd^*\dd v^\sharp = \Delta_g v^\sharp. 
\end{align*}
The last equality follows from the definition of Laplace--Beltrami operator $$\Delta_g = \dd \dd^* +\dd^*\dd$$ and that $v$ is divergence-free (\emph{i.e.}, $\dd^*v^\sharp  =0$). Thus we obtain the following expression for $\sol$:
\begin{equation}\label{biot savart, expression}
\sol \omega = \left[ \Delta_g^{-1} \left(\star \dd \omega\right) \right]^\flat,
\end{equation}
where $\Delta_g^{-1}$ is the inverse of Laplace--Beltrami operator acting on $\Omega^1(\Sigma)$ and restricted to the cohomology class of $(v_0)^\sharp$.

\begin{example}
For  $\Sigma = \tor$, the Biot--Savart operator can be expressed explicitly as follows on the Fourier side; here $v_0=0$. 
\begin{align*}
\widehat{\sol\omega}(\xi) = -\2 \frac{\widehat{\omega}(\xi) \begin{bmatrix}
 0 &-1\\
 1&0
\end{bmatrix}\xi}{|\xi|^2}.
\end{align*}
See \cite[Equation~(3.6)]{s}.
\end{example}
For a general closed orientable surface $(\Sigma,g)$, we make use of Green's operators to find $\sol$. See, for instance, the classical work of de Rham \cite[p.134, Chapter V, Theorem~23]{dr}. There exist linear operators $\HH$ and $\GG$ on $\Omega^1(\Sigma)$, such that $\HH$ is the projection onto $\harm$, the space of harmonic $1$-forms on $\Omega$, and $\GG$ satisfies $$\GG\Delta_g = \Delta_g\GG = \1 - \HH.$$ Both $\HH$ and $\GG$ commute with $\dd$, $\dd^*$, and $\star$. The projection $\HH$ has a $C^\infty$-kernel, and the kernel $G(x,y): \Sigma \times \Sigma \to \R$ for $\GG$ satisfies is smooth off-diagonal and of order $\mathcal{O}(|\log r|)$ for $r = d_g(x,y)$. In our setting, since $v^\sharp$ is cohomologous to $v_0^\sharp$ and they are both $\dd^*$-free, so $$v^\sharp - v_0^\sharp\in\harm.$$ Thus we have
\begin{align*}
\GG(\star \dd \omega) = ( \1-\HH) v^\sharp = v^\sharp - \HH v_0^\sharp,
\end{align*}
which leads to
\begin{equation}\label{sol, expression}
v = \sol \omega = \left(\GG\left(\star \dd \omega\right) + \HH v_0^\sharp\right)^\flat.
\end{equation}

Let us determine the principal symbol of $\sol$. It is well-known that the symbol for $\dd$ is 
\begin{equation}\label{d symbol}
\sigma(\dd) = 2\pi\2\xi\wedge\bullet
\end{equation}
on any differentiable manifold. When specialising to  compact surface, we denote by $\J$ the almost complex structure on $(\Sigma,g)$, where $(\Sigma,g)$ is viewed as a complex manifold. Then the principal symbol for $\star \dd$ is equal to (the left-multiplication by) $2\pi\2 \J\xi$. On the other hand, it is well-known that the principal symbol is multiplicative:
\begin{equation}\label{ppl}
\ppl(AB) = \ppl(A)\, \ppl(B)
\end{equation}
for pseudo-differential operators between vector bundles over a given compact manifold. See, \emph{e.g.}, \cite[p.84, Proposition~3.1.6]{ni}. In our case the bundle is $T^*\Sigma$, and the multiplication on the right-hand side of \eqref{ppl} is taken with respect to the metric. Recall that Green's operator $\GG$ is a parametrix for the Laplace--Beltrami $\Delta_g$:
\begin{align*}
\GG\circ\Delta_g = \Delta_g \circ \GG = \1 +\HH.
\end{align*}
But the projection $\HH: \Omega^1(\Sigma) \to \harm$ is smooth; \emph{i.e.}, of class $\Psi^{-\infty}$ as a pseudo-differential operator. Hence
\begin{equation*}
\ppl(\GG)(x,\xi) = \frac{1}{\ppl(\Delta_g)(x,\xi)} = \frac{-1}{(2\pi)^2 |\xi|_g^2}\qquad \text{for all } (x,\xi) \in T^*\Sigma.
\end{equation*}
By virtue of \eqref{sol, expression}, we obtain by applying \eqref{ppl} and $\HH \in \Psi^{-\infty}$ once more that
\begin{equation}\label{symbol for Biot-Savart}
\ppl(\sol)(x,\xi) = -\frac{\2}{2\pi} \frac{\J\xi}{|\xi|_g^2} \qquad \text{for all } (x,\xi) \in T^*\Sigma.
\end{equation}
Here and hereafter, we view $\J$ as a section of the endomorphism bundle ${\rm End}(T^*\Sigma)$ instead of a $(1,1)$-tensor through the duality ${\rm End}(T^*\Sigma) \cong T^*\Sigma \otimes T\Sigma$.

As we have fixed the cohomology class, the $\Delta_g$ is invertible with inverse $\sol$, hence its Fredholm index is zero.

\section{Proof of Fredholm quasiruledness}\label{section: proof}

\subsection{The quantity $W$}

With the toolbox of paradifferential calculus on $(\Sigma,g)$ in  \S\ref{section: paradiff} at hand, we  prove Theorem~\ref{thm} by adapting the arguments in  \cite[\S 3]{s}. When $\Sigma = \tor$, Shnirelman \cite{s} considered the quantity 
\begin{align}\label{W, flat}
W(\varphi)(x,t) := \Pi_{\dd\left(\varphi_t^{-1}\right)} \left(\varphi_t - \1\right)(x),
\end{align}
where $\varphi_t$ is the integral flow of the Eulerian velocity $v$. 
This $W$ is a good quantity with suitably controlled divergence and curl. One recovers $\varphi_t$ from the curl of $W$ via the Biot--Savart operator. 

\begin{remark}\label{rem: vector product}
Here and hereafter, for matrix-valued function $A = \left\{A^i_j\right\}_{q \leq i \leq \ell;\,1 \leq j \leq m}$ and  vector-valued function $V = \left(V^1,\ldots, V^m\right)^\top$, the paraproduct $\Pi_AV$ is the vector-valued function given by $
\left(\Pi_AV\right)^i := \sum_{j=1}^m \Pi_{A^{i}_j} V^j$ for each $i \in \{1,2,\ldots,\ell\}$. 
\end{remark}

We look for an analogue of $W$ in \eqref{W, flat} on a generic surface $(\Sigma,g)$. Note, however, that $\dd\varphi_t^{-1} \in \End(T\Sigma)$ and $\varphi_t \in \sdiff(\Sigma,g)$, so the paraproduct $\Pi_{\dd\varphi_t^{-1}}\varphi_t$ does not make sense unless $\Sigma$ is flat. We overcome this issue by first embedding $(\Sigma,g)$ isometrically into $\R^m$, and then extending $\varphi_t$ to a map $\R^m \to \R^m$ by Whitney extension.

\begin{definition}\label{definition: whitney}
Given a smooth, connected, orientable closed surface $(\Sigma,g)$  and volume-preserving diffeomorphisms $\{\varphi_t\} \subset \hsdiff(\Sigma,g)$ for $s >2$. Let $\iota: (\Sigma,g) \emb \R^m$ be a $C^\infty$-isometric embedding for some $m \geq 3$. Note that $\iota \circ \varphi_t: \Sigma \to \R^m$ is an vector-valued $H^s$-function. 
\begin{itemize}
\item
Take $\Phi_t: \R^m \to \R^m$ to be any $H^s$-extension of $\iota \circ \varphi_t$.
\item
Put $\p\Phi_t(z)\slash \p t = V(\Phi_t(z))$ for each $z\in \R^m$. 
\item
Take $\Psi_t : \R^m \to \R^m$ to be any $H^s$-extension of $\iota \circ \varphi_t^{-1}$. 
\item
Set \begin{equation}\label{W, surface}
W_t := \Pi_{D\Psi_t} \Phi_t: \R^m \longrightarrow \R^m.
\end{equation}
\end{itemize}
\end{definition}

The existence of $\iota$ follows from Nash's $C^k$-isometric embedding theorem \cite{n}, and the existence of $\Phi_t$ and $\Psi$ follows from the Besov (hence Sobolev) version of the Whitney extension theorem; see Jonsson--Wallis  \cite{jw}. Here $\Phi_t$ and $\Psi_t$ are $H^s$-vectorfields on $\R^m$ such that 
\begin{align*}
\Phi_t \circ \iota = \iota \circ \varphi_t\quad \text{ and } \quad \Psi_t \circ \iota = \iota \circ \varphi_t^{-1}\quad \text{ almost everywhere on $\Sigma$}.
\end{align*}
The notion of almost everywhere is understood with respect to $\dvg$. The Riemannian volume measure $\dvg$ is Ahlfors 2-regular and supported in $\Sigma$, embedded via $\iota$ as a closed set in $\R^m$; hence, the assumptions for \cite[p.146, Main Theorem]{jw} are verified. One may consider $\iota \circ \varphi_t = \left[ (\iota \circ \varphi_t)^1, \ldots, (\iota \circ \varphi_t)^m \right]^\top$ and apply to each component the Whitney extension to obtain $\Phi_t$, and similarly for $\Psi_t$. In Euclidean co-ordinates $$(W_t)^j =\sum_{i=1}^m \Pi_{D_i (\Psi_t)^j} (\Phi_t)^i \qquad \text{ for each $j \in \{1,2,\ldots,m\}$},$$ thanks to Remark~\ref{rem: vector product}. In addition, one has
\begin{equation}\label{V and v}
V \circ \dd\iota = \dd \iota \circ v \qquad \text{ on } \G(T\Sigma).
\end{equation}

\subsection{Time derivative of $W$} 
The following is an adaptation of the arguments on \cite[pp.S393--S394]{s}. As a caveat, we remark that $\Psi_t = \Phi_t^{-1}$ is not required to hold on the whole space $\R^m$, which maybe overdetermined in general. We have only defined $\Phi_t$ to be an extension of $\varphi_t$ and $\Psi_t$ an extension of $\varphi_t^{-1}$. That is, $\Psi_t = \Phi_t^{-1}$ is only imposed on $\iota(\Sigma)$.

\begin{lemma}\label{lem: dt W}
When restricted to $T[\iota(\Sigma)]$, the time derivative of $W_t$ satisfies
\begin{align*}
\frac{\p W_t}{\p t}  = \iota_\#\left\{\left(\Pi_{\dd\varphi_t^{-1}}\circ \KK_{\varphi_t} \circ \sol \circ \KK_{\varphi_t^{-1}}\right) (\omega)\right\} + {\rm Remainder},
\end{align*}
where $ {\rm Remainder} \in {\bf H}^{\sigma-1}(T[\iota(\Sigma)])$ for some $\sigma > s$, and $\Pi,\KK$ denote respectively the paraproduct and paracomposition as in \S\ref{section: paradiff}. 
\end{lemma}

\begin{proof}
We apply the Leibniz rule to $W_t$ defined in \eqref{W, surface} to get
\begin{align*}
\frac{\p}{\p t} W_t &= \Pi_{\frac{\p D\Psi_t}{\p t} } \Phi_t + \Pi_{D\Psi_t} \frac{\p \Phi_t}{\p t} = J_1 + J_2.
\end{align*}

For the second term, since $\varphi_t = \iota^{-1} \circ \Phi_t \circ \iota$ on $\iota(\Sigma)$ and $\varphi_t$ is the integral flow of $v$:
\begin{align*}
\frac{\p}{\p t} \varphi_t(x) = v\left(\varphi_t(x)\right),
\end{align*}
we have 
\begin{align}\label{J2}
J_2 &= \Pi_{D\Psi_t} \left\{ \dd\iota \circ (v \circ \varphi_t) \circ \dd\iota^{-1} \right\}\nonumber\\
&= \Pi_{D\Psi_t} (V\circ\Phi_t).
\end{align}
Here, we make use of the identification of  tangent spaces via $\dd \iota: T\Sigma \to T\R^m$ and $\dd\iota^{-1}: T[\iota(\Sigma)] \to T\Sigma$, as well as \eqref{V and v}.

For the first term, as $\Psi_t  = \iota \circ \varphi_t^{-1} \circ \iota^{-1}$ on $\iota(\Sigma)$, we may compute as follows: 
\begin{align*}
\left(\frac{\p}{\p t} D\Psi_t\right) &= \dd\iota \circ \left( \frac{\p}{\p t }\,\dd\varphi_t^{-1} \right) \circ \dd\iota^{-1}\\
&= - D\Psi_t \bullet \left\{\dd\iota\circ\dd\left(\frac{\p}{\p t}\varphi_t \right)\circ \dd\iota^{-1}\right\}\bullet D\Psi_t \\
&= -D\Psi_t \bullet \left\{ \dd\iota \circ \dd\left(v\circ\varphi_t\right) \circ \dd\iota^{-1} \right\}\bullet D\Psi_t\\
&= -D\Psi_t \bullet \left\{\left[\dd\iota \circ  ((\dd v)\circ\varphi_t )\circ\dd\iota^{-1}\right] \bullet \left[\dd\iota\circ \dd\varphi_t \circ \dd\iota^{-1}\right] \right\}\bullet D\Psi_t\\
&= -D\Psi_t \bullet\left[\dd\iota \circ  ((\dd v)\circ\varphi_t) \circ\dd\iota^{-1}\right]. 
\end{align*}
In the last line we make use of the identities $$\dd \iota \circ \dd\varphi_t \circ \dd \iota^{-1} = \dd (\iota \circ \varphi_t \circ \iota^{-1})=D\Phi_t = (D\Psi_t)^{-1}.$$
Thus, on the tangent bundle of $\iota(\Sigma)$, it holds that
\begin{align*}
J_1 &= -  \Pi_{D\Psi_t \bullet\left[\dd\iota \circ  ((\dd v)\circ\varphi_t) \circ\dd\iota^{-1}\right]} \left\{\iota \circ \varphi_t \circ \iota^{-1}\right\} \nonumber\\
&= -  \Pi_{D\Psi_t \bullet\left[\dd\iota \circ  ((\dd v)\circ\varphi_t) \circ\dd\iota^{-1}\right]} \Phi_t.
\end{align*}
Moreover, by \eqref{V and v} we have $$\dd\iota \circ ((\dd v)\circ\varphi_t)\circ \dd\iota^{-1} = DV\circ\Phi_t.$$ So we deduce that 
\begin{equation}\label{J1}
J_1 = - \Pi_{D\Psi_t \bullet [DV \circ \Phi_t]} \Phi_t.
\end{equation}

Putting together \eqref{J1} and \eqref{J2}, we conclude that
\begin{equation}\label{pW pt}
\frac{\p}{\p t} W_t = \Pi_{D\Psi_t} (V\circ \Phi_t) - \Pi_{D\Psi_t\bullet[DV\circ\Phi_t]} \Phi_t. 
\end{equation}
We point out one possible source of confusion here --- take, for instance, the expression $\dd \iota \circ (v \circ\varphi_t) \circ \dd\iota^{-1}$ in the computation for $J_2$ above. In the strict sense, the second composition is different from the other two; one should view it via $v\circ\varphi_t (x)= v(\varphi_t(x)) := v\big|_{\varphi_t(x)}$. This agrees with the understanding of  vectorfield $v \in \G(T\Sigma)$ as a map $v: \Sigma \to T\Sigma$. In the end, all the compositions in \eqref{pW pt} are of this kind.

Now, by virtue of the convention in Remark~\ref{rem: vector product}, we make use of the paracomposition formula~\eqref{paracomposition, simplified} to infer that
\begin{align*}
\Pi_{D\Psi_t} (V \circ \Phi_t) = \Pi_{D\Psi_t}\big\{ \mathcal{K}_{\Phi_t} V + \Pi_{DV\circ\Phi_t}\Phi_t \big\}.
\end{align*}
Moreover, the product of paraproduct satisfies
\begin{align}\label{cancellation, paraprod}
\left(\Pi_{D\Psi_t} \circ \Pi_{DV \circ \Phi_t} - \Pi_{D\Psi_t \bullet [DV\circ\Phi_t]}\right)\Phi_t \in H^{s+1}
\end{align}
for $\Phi_t \in H^s(\R^m, \R^m)$. It is crucial to have the paraproduct structure   in \eqref{cancellation, paraprod}, which carries only the singularity of $\Phi_t$ and smears out the singularities of $D\Psi_t$ and $DV$ in the subscripts.

Thus, we may further continue \eqref{pW pt} as
\begin{align*}
\frac{\p W_t}{\p t}  = \Pi_{D\Psi} \left(\mathcal{K}_{\Phi_t} V\right)   + {\rm Remainder},
\end{align*}
with $ {\rm Remainder} \in {\bf H}^{\sigma-1}(T[\iota(\Sigma)])$ for some $\sigma > s$.  From the definition of $\Phi_t$, $\Psi_t$, and $V$ as extensions of $\varphi_t$, $\varphi_t^{-1}$, and $v$, respectively, as well as the fact that $\iota: (\Sigma,g) \emb \R^m$ is an isometric embedding (see Remark~\ref{rem: isom emb and para}), we conclude that
\begin{align*}
\frac{\p W_t}{\p t} \mres T[\iota(\Sigma)] = \iota_\#\left\{\left(\Pi_{\dd\varphi_t^{-1}} \KK_{\varphi_t} \right) v\right\} + {\rm Remainder},
\end{align*}
where $ {\rm Remainder} \in {\bf H}^{\sigma-1}(T[\iota(\Sigma)])$ for some $\sigma > s$.

Finally, let us expand the principal term on the right-hand side of the above equality via the definition of $v$. Recall that
\begin{align*}
\frac{\p}{\p t} \varphi_t = v \circ \varphi_t = \sol\left[ \omega \circ \varphi_t^{-1} \right] \circ\varphi_t.
\end{align*}
So
\begin{align*}
\Pi_{\dd\varphi_t^{-1}} \KK_{\varphi_t}v &= \Pi_{\dd\varphi_t^{-1}} \KK_{\varphi_t}\sol\left[ \omega \circ \varphi_t^{-1} \right].
\end{align*}
A further application of the paracomposition~\eqref{paracomposition}  gives us
\begin{align*}
\omega \circ \varphi_t^{-1} &=  \KK_{\varphi_t^{-1}} \omega + \Pi_{\dd\omega \circ \varphi_{t}^{-1}}\varphi_t^{-1} + {\rm Remainder\,(I)}\\
&= \KK_{\varphi_t^{-1}} \omega +{\rm Remainder\,(II)}
\end{align*}
for ${\rm Remainder\,(I),\,(II)} \in {\bf H}^{\sigma-1}(\Sigma,g)$ for some $\sigma > s$. Note here that $\Pi_{\dd\omega \circ \varphi_{t}^{-1}}\varphi_t^{-1}$ carries only the singularities of $\varphi_t^{-1}$, thus is of regularity $H^s$. The proof is now complete.   \end{proof}

\subsection{Div and curl of $\p W_t \slash \p t$}

In the previous subsection we investigated the vectorfield
\begin{equation}\label{Ut}
U_t := \iota^\#\left\{\frac{\p W_t}{\p t}\right\} \in {\bf H}^{s-1}(\Sigma;T\Sigma).
\end{equation}
Recall that we first defined $W_t: \R^m \to \R^m$ as a flow on $\R^m$ which, by construction, restricts to a one-parameter family of diffeomorphisms on $\iota(\Sigma)$ along the integral curve of $\varphi_t$. Thus $U_t$ in \eqref{Ut} is indeed a well-defined vectorfield on $\Sigma$. Moreover, since $\iota$ is a smooth isometric embedding and $W_t$ has the same $H^s$-regularity as $\varphi_t$, we have $U_t \in {\bf H}^{s-1}(\Sigma;T\Sigma)$.

Now, following \cite[\S 3]{s}, let us focus on the divergence and curl of $U_t$; that is,
\begin{align*}
\dd^* U_t  \in {\bf H}^{s-2} (\Sigma;\R)\quad \text{ and } \quad \star\dd\left[\left(U_t\right)^\sharp\right] \in  {\bf H}^{s-2}(\Sigma,\R). 
\end{align*}
We shall conclude the proof for our main Theorem~\ref{thm} by analysing these two quantities. More precisely, we prove that if the curl of $U_t$ is Fredholm and if the divergence of $U_t$ equals zero modulo a remainder of higher regularity, then $\Exp$ is Fredholm quasiruled.
\begin{proposition}\label{proposition: intermediate}
Let $(\Sigma,g)$ be a closed, connected, orientable surface, and let $s>2$. Define
\begin{align}\label{B}
\mathcal{B}: {\bf H}^{s-1}(\Sigma;\R)\times {\bf H}^{s-1}(\Sigma;\R)\, &\longrightarrow\, {\bf H}^{s-1}(\Sigma;\R),\nonumber\\
(\zeta,\omega) \,&\longmapsto \, \star\dd\left[\left(\Pi_{\dd\Xi^{-1}}\circ \KK_{\Xi} \circ \sol \circ \KK_{\Xi^{-1}}\right) (\omega)\right]^\sharp
\end{align}
where $$\Xi = \Xi_t = \int_0^t \sol\left(\omega\circ\Xi_\tau^{-1}\right)\circ\Xi_\tau \,\dd\tau.$$ Suppose that $$\widetilde{\mathcal{B}}(\omega):=\mathcal{B}(\omega,\omega)$$ is a Fredholm mapping of index zero. Suppose also that $$\dd^* U_t \in {\bf H}^{\sigma-2}(\Sigma;\R)$$ for some $\sigma >s$. Then $\Exp:T_{\id} \hsdiff(\Sigma,g) \to \hsdiff(\Sigma,g)$ is  a Fredholm quasiruled map of index zero.
\end{proposition}

For notational convenience we shall sometimes suppress the variable $t$, when all the relevant identities/estimates are kinematic, \emph{i.e.}, holding pointwise in $t$.

\begin{proof}
We first notice by Lemma~\ref{lem: dt W} that 
\begin{align*}
\star\dd\left[\left(U_t\right)^\sharp\right] \approx \star\dd\left[\left(\Pi_{\dd\varphi_t^{-1}}\circ \KK_{\varphi_t} \circ \sol \circ \KK_{\varphi_t^{-1}}\right) (\omega)\right]^\sharp
\end{align*}
modulo a remainder term of higher regularity (namely, $H^{\sigma-2}$ for some $\sigma > s$). As $\Sigma$ is compact, the Fredholm property is invariant under such higher regularity perturbations by the Rellich lemma. Moreover, recall that $\frac{\p}{\p t}\varphi_t = \sol \left(\omega\circ\varphi_t^{-1}\right)\circ\varphi_t$, which give us
\begin{align*}\widetilde{\mathcal{B}}(\omega):=\mathcal{B}(\omega,\omega) = \star \dd\left(U_t \right)^\sharp = \star\dd\left[\iota^\#\left(\frac{\p W_t}{\p t} \right)\right]^\sharp + {\rm Remainder}
\end{align*}
with ${\rm Remainder} \in {\bf H}^{\sigma-2}(\Sigma;\R)$. This space embeds compactly into ${\bf H}^{s-2}(\Sigma;\R)$, so the remainder has no effect on Fredholm index.

Consider now the Bochner integral 
\begin{align}\label{I, def}
\mathcal{I}(t) := \int_0^t \widetilde{\mathcal{B}} (\omega(\tau, \bullet))\,\dd\tau. 
\end{align}
This defines a function for each $t$, which equal to the curl of $W_t$ when restricted to $\iota(\Sigma)$. More precisely, denoting by $D$ the Euclidean gradient, one has
\begin{align}\label{I, Bochner}
\mathcal{I}(t) &= \star \circ \iota^\#\left\{ D\left[(W_t)^\sharp\right]\right\}\nonumber\\
&= \star \circ \iota^\#\left\{ D\left[\left( \Pi_{D\Psi_t}\Phi_t \right)^\sharp\right]\right\}  \,\in\,{\bf H}^{s-2}(\Sigma;\R).
\end{align}
Here, $D\left[(W_t)^\sharp\right]$ is a differential 2-form on $\R^m$, whose restriction to $\iota(\Sigma)$ is also a 2-form on this  2-dimensional submanifold. Recall also that $\Psi_t$, $\Phi_t$ are extensions of $\varphi_t^{-1}$ and $\varphi_t$, respectively. The Hodge star in \eqref{I, Bochner} is a mapping $\star: {\bf H}^{s-2}\left(\Sigma; \bigwedge^2T^*\Sigma\right) \to {\bf H}^{s-2}(\Sigma;\R)$.

We can recover $\varphi_t \in \hsdiff(\Sigma,g)$ from $\mathcal{I}(t)$ as follows. Notice from \eqref{I, Bochner} that
\begin{align*}
\iota_\#\varphi_t = \Phi_t \circ\iota = \Pi_{\dd\Psi}\circ\iota_\# \circ \overline{\sol} \left(\mathcal{I}(t)\right).
\end{align*}
Here $\overline{\sol}$ is the inverse of $\star \circ \iota^\#\circ D \circ \sharp$ in \eqref{I, Bochner}, which is nothing but the curl/rot operator on $\iota(\Sigma)$. So $\overline{\sol}$ is the corresponding Biot--Savart operator, and is well-defined thanks to the assumption  $\dd^* U_t \in {\bf H}^{\sigma-2}(\Sigma;\R)$ for  $\sigma >s$. More precisely, let us solve for $\varsigma := \overline{\sol}(\mathcal{I}(t))$ such that 
\begin{equation*}
\overline{d} \varsigma = \mathcal{I}(t)\qquad \text{ and } \qquad \overline{d^*} \varsigma = 0,
\end{equation*}
where $\overline{d} = \iota^\#\circ D$ is the exterior differential on $\iota(\Sigma)$, and $\overline{d^*}$ is the corresponding codifferential with respect to the metric $\iota_\# g$. Moreover, we fix any harmonic $1$-form $\varsigma_0 \in {\bf Harm}^1(\iota(\Sigma))$ and require $\varsigma$ and $\varsigma_0$ to be cohomologous. Then, denoting by $\overline{\mathscr{G}} $ the Green's matrix for $(\iota(\Sigma),\iota_\#g)$, by $\overline{\Delta}$ the corresponding Laplace--Beltrami operator, and by $\mathscr{H}_{\zeta_0}$ the projection onto the cohomology class of $\varsigma_0$, one obtains
\begin{align*}
\varsigma = \overline{\mathscr{G}} \left\{ \overline{\Delta} \int_0^t U_\tau\,\dd \tau - \overline{\dd}\,\overline{\dd^*} \int_0^t U_\tau\,\dd\tau \right\}  + \mathscr{H}_{\varsigma_0}(\varsigma).
\end{align*}
See \S\ref{section: BS} for details. The projection $\mathscr{H}_{\zeta_0}$ is smooth and the Green's operator $\overline{\mathscr{G}}$ is of order $-2$. In addition, by assumption we have $\overline{\dd^*} U_t \in {\bf H}^{\sigma-2}(\Sigma;\R)$ for $\sigma >s$. Thus 
\begin{align*}
&\mathcal{I}(t) = \overline{\dd}\int_0^t U_\tau\,\dd \tau \in{\bf H}^{s-2}(\Sigma;\R)\\
&\qquad\Longrightarrow \, \varsigma = \overline{\sol}(\mathcal{I}(t))  \in{\bf H}^{s-1}(\Sigma;T\Sigma) + {\bf H}^{\sigma-1}(\Sigma;T\Sigma) \quad\text{for some $\sigma > s$}.
\end{align*}

The discussions in the previous paragraph and the definition of $\mathcal{I}(t)$ in  \eqref{I, def} yield that 
\begin{align}\label{x, new}
\varphi_t &= \iota^\# \circ \Pi_{D\Psi} \circ \iota_\#\circ \overline{\sol} \left(\mathcal{I}(t)\right) \nonumber\\
&= \iota^\# \circ \Pi_{D\Psi} \circ \iota_\#\circ \overline{\sol} \left( \int_0^t \widetilde{\mathcal{B}} \left(\omega(\tau,\bullet)\right)\,\dd\tau\right),
\end{align}
where $\mathcal{I}(t) \in {\bf H}^{s-2}(\Sigma;\R)$. By construction of the Biot--Savart operator $\overline{\sol}$ and the paraproduct (see Theorem~\ref{thm: para}), we find that $\iota^\# \circ \Pi_{D\Psi} \circ \iota_\# \circ \overline{\sol}$ maps ${\bf H}^{s-2}(\Sigma;\R)$ continuously into $\hsdiff(\Sigma,g)$, and each of the four mappings in the composition are Fredholm of index zero. 

Therefore, if $\widetilde{\mathcal{B}} \left(\omega\right)$ is Fredholm of index zero (hence so is its Bochner integral $\mathcal{I}(t)$), then for each $t$ the mapping
\begin{align*}
\G_t: {\bf H}^{s-2}(\Sigma;\R) \,&\longrightarrow\, \hsdiff(\Sigma,g),\\
\omega \,&\longmapsto\, \varphi_t
\end{align*}
is also Fredholm of index zero. Then
\begin{equation*}
\Exp (v) \equiv \G_1 \left( \star\dd\left(v^\sharp\right) \right) :\, \left[v_0\right] \subset T_{\id} \hsdiff(\Sigma,g)\, \longrightarrow \, \hsdiff(\Sigma,g)
\end{equation*}
is Fredholm of index zero too. We restrict the domain of $\Exp$ to $$[v_0] := \left\{v \in T_{\id} \hsdiff(\Sigma,g):\, v^\sharp \text{ and } (v_0)^\sharp \text{ are cohomologous} \right\},$$ in order to ensure the invertibility of the differential $\dd$. Moreover, as $\widetilde{\mathcal{B}} \left(\omega\right)$ is Fredholm quasilinear, $\Exp$ is Fredholm quasiruled.  The proof for the proposition is now complete. \end{proof}

Therefore, to establish the main Theorem~\ref{thm}, it remains to verify the two assumptions in Proposition~\ref{proposition: intermediate}. This shall be carried out in the remaining two subsections.

\subsection{Divergence of $U_t$}

We prove the following
\begin{lemma}\label{lem: divergence of W} Let $U_t$ be as in \eqref{Ut}. Then $\dd^* U_t$ takes values in a compact subset of ${\bf H}^{s-2}(\Sigma;\R)$.
\end{lemma}

\begin{proof}

By Lemma~\ref{lem: dt W} we have
\begin{align*}
\dd^* U_t = \dd^*\left\{\left(\Pi_{\dd\varphi_t^{-1}}\circ \KK_{\varphi_t} \circ \sol \circ \KK_{\varphi_t^{-1}}\right) (\omega)\right\} + {\rm Remainder},
\end{align*}
with ${\rm Remainder} \in {\bf H}^{\sigma-2}(\Sigma;\R)$ for some $\sigma > s$.

Observe also that
\begin{align*}
 \dd^*\left\{\left(\Pi_{\dd\varphi_t^{-1}}\circ \KK_{\varphi_t} \circ \sol \circ \KK_{\varphi_t^{-1}}\right) (\omega)\right\} = \Pi_{\dd\varphi_t^{-1}}\left\{\dd^*\left( \KK_{\varphi_t} \circ \sol \circ \KK_{\varphi_t^{-1}}(\omega) \right)\right\} + {\rm Remainder},
\end{align*}
with ${\rm Remainder} \in {\bf H}^{s-1}(\Sigma;\R)$, since the commutator $\left[\dd^*,\pi_{\dd\varphi_t^{-1}}\right]$ is a pseudodifferential operator of order $0$.

Next, thanks to Corollary~\ref{cor: paracomposition} with $\varrho = s-1$ therein,
\begin{align*}
 \KK_{\varphi_t} \circ \sol \circ \KK_{\varphi_t^{-1}}(\omega) = \left[\sol\circ\KK_{\varphi_t^{-1}}(\omega)\right]\circ\varphi_t + {\rm Remainder},
\end{align*}
with ${\rm Remainder} \in {\bf H}^{\sigma-1}(\Sigma;\R)$ for some $\sigma > s$. Indeed,   the paracomposition leads to a term $$[{\rm BAD}] := \Pi_{ \dd\left[ \sol \circ \KK_{\varphi_t^{-1}}(\omega) \right] \circ\varphi_t}\varphi_t;$$ for simplicity we abbreviate it as $\Pi_A\varphi_t$. However, $\omega \in {\bf H}^{s-1}$ and hence $A \in {\bf H}^{s-2}$, for which we cannot directly apply  the estimate in Theorem~\ref{thm: para}: it needs $A \in L^\infty$ thereof, but ${\bf H}^{s-2}$ does not embed in $L^\infty$ when merely assuming $s>2$.

To this end, we show by means of hands-on estimates that for $s=2+\e$, $\varrho = 1+2\e$ for any $\e \in ]0,1[$, we can indeed bound the $H^{\varrho}$-norm of  $[{\rm BAD}]$. Note that 
\begin{align*}
[{\rm BAD}] = (\1 -\Delta_g)^{-\varrho \slash 2} \left( \bra A,\,  \Delta^{\varrho \slash 2} \varphi_t \ket \right).
\end{align*}
Then, by Sobolev embedding, we have $A \in {\bf H}^{\e} \emb L^{\frac{2}{1-\e}}$ and $\Delta^{\varrho \slash 2} \varphi_t \in {\bf H}^{1-\e} \emb L^{\frac{2}{\e}}$, so Cauchy--Schwarz gives us $\bra A,\,  \Delta^{\varrho \slash 2} \varphi_t \ket \in L^2$. (Here $\bra\bullet,\bullet\ket$ is a quadratic function with bounded coefficient.) The $H^\varrho \emb H^{\varrho}$ boundedness of the Riesz transform shows that $[{\rm BAD}]$ is bounded in $H^{\varrho} = H^{1+2e}$, which is compactly embedded into $H^{s-1} = H^{1+\e}$ by the Rellich lemma.

Finally, by construction of the  Biot--Savart operator in \S\ref{section: BS}, $\sol$ takes values in the $\dd^*$-free part of differential 1-forms. So we have
\begin{align*}
\dd^* \left\{\left[\sol\circ\KK_{\varphi_t^{-1}}(\omega)\right]\circ\varphi_t \right\} = 0. 
\end{align*}

In summary, the above computations yield that $\dd^* U_t = 0 + {\rm Remainder} \in {\bf H}^{\sigma-2}(\Sigma;\R)$ for some $\sigma > s$. The proof is complete by the Rellich lemma. \end{proof}

\subsection{$\widetilde{\mathcal{B}}$ is Fredholm}

Finally, we prove that 
\begin{lemma}\label{lem: final}
The mapping $\widetilde{\mathcal{B}}: {\bf H}^{s-1}(\Sigma;\R) \to {\bf H}^{s-1}(\Sigma;\R)$ defined in Proposition~\ref{proposition: intermediate} is Fredholm of index zero.
\end{lemma}

\begin{proof}
The previous arguments show that
\begin{equation*}
\widetilde{\mathcal{B}}(\omega)= \underbrace{\star\dd\left[\left(\Pi_{\dd\varphi_t^{-1}}\circ \KK_{\varphi_t} \circ \sol \circ \KK_{\varphi_t^{-1}}\right) (\omega)\right]^\sharp}_{:=\,\main} +\, {\rm Remainder},
\end{equation*}
where ${\rm Remainder} \in {\bf H}^{\sigma-1}(\Sigma;\R)$ for $\sigma > s$, thus having no effect on the Fredholm index.   $\main$  is a pseudodifferential operator of order zero; this is because  $\Pi$, $\KK$ are of order 0, $\sol$ is of order $-1$, and $\dd$ is of order $1$. Thus, to prove the thesis, it suffices to check that $\main$ has a nonvanishing principal symbol, viewed as a function defined on the cotangent bundle $T^*\Sigma$.

Recall that the principal symbol for the Biot--Savart operator $\sol$ has already been computed in \eqref{symbol for Biot-Savart}. Moreover, the conjugation with $\KK_{\varphi_t}$ contributes a factor of $\left(\dd\varphi_t^*\right)^{-1}$  to the symbol, where the asterisk denotes the operator adjoint. Using once again the symbol for $\dd$ (see  \eqref{d symbol}) and its expression on $(\Sigma,g)$ via the almost complex structure $\J$, we have 
\begin{align}\label{symbol for Gamma}
\ppl(\main)(x,\xi) &= {2\pi\2} \bra\J\xi,\, \dd\varphi_t^{-1} \left\{-\frac{\2}{2\pi}\cdot\frac{\J \left(\dd\varphi_t^*\right)^{-1} \xi}{\left|\left(\dd\varphi_t^*\right)^{-1}\xi\right|^2_g}\right\}\ket_g\nonumber\\
&= \bra \J\xi,\,\dd\varphi_t^{-1}\frac{\J \left(\dd\varphi_t^*\right)^{-1} \xi}{\left|\left(\dd\varphi_t^*\right)^{-1}\xi\right|^2_g}\ket_g.
\end{align}
 
It is clear from \eqref{symbol for Gamma} that the principal symbol for $\main$ is $0$-homogeneous in the fibre variable $\xi$. Then it suffices to show the \emph{ellipticity} of $\main$; that is, $\ppl(\main)(x,\xi)=0$ implies that $\xi = 0$.  

For this purpose, we infer from \eqref{symbol for Gamma} and the identity $$\left(\dd\varphi_t^{-1}\right)^* = \left(\dd\varphi_t^*\right)^{-1}$$ that
\begin{align*}
\ppl(\main)(x,\xi) = \frac{\bra \left(\dd\varphi_t^*\right)^{-1}\J\xi,\,\J\left(\dd\varphi_t^*\right)^{-1}\xi\ket_g}{\left|\left(\dd\varphi_t^*\right)^{-1}\xi\right|^2_g}.
\end{align*}
Since $(\Sigma,g)$ is a compact orientable Riemannian manifold of dimension two, it is K\"{a}hler; hence, there exists a symplectic $2$-form $\omega \in \Omega^2(\Sigma)$ such that $(g,\J,\omega)$ is a compatible triple. As a consequence, one has \begin{align*}
\ppl(\main)(x,\xi) = \frac{\omega\left(\left(\dd\varphi_t^*\right)^{-1}\xi,\,\left(\dd\varphi_t^*\right)^{-1}\J\xi\right)}{\left|\left(\dd\varphi_t^*\right)^{-1}\xi\right|^2_g}.
\end{align*}
That is,
\begin{align*}
\ppl(\main)(x,\xi) = \frac{\widetilde\omega\left(\xi,\,\J\xi\right)}{\left|\left(\dd\varphi_t^*\right)^{-1}\xi\right|^2_g}
\end{align*}
where $\widetilde{\omega}$ is the pushforward $2$-form: $$\widetilde{\omega} := \left(\varphi_t\right)_{\#}\omega \in \Omega^2(\Sigma).$$ This is well-defined as $\varphi$ is a diffeomorphism, and $\widetilde{\omega}$ is also a symplectic form (\emph{i.e.}, closed and nondegenerate). Then $\widetilde{\omega}(\bullet,\J\bullet)$ defines a Riemannian metric on $\Sigma$; call it $\widetilde{g}$. Thus
\begin{align*}
\ppl(\main)(x,\xi) &= \frac{|\xi|^2_{\widetilde{g}}}{\left|\left(\dd\varphi_t^*\right)^{-1}\xi\right|^2_g} \\
&= \frac{\widetilde{g}(\xi,\xi)}{g\left(\left(\dd\varphi_t^*\right)^{-1}\xi,\,\left(\dd\varphi_t^*\right)^{-1}\xi\right)}.
\end{align*}
This proves that $\main$ is elliptic; hence, $\widetilde{\mathcal{B}} $ is Fredholm of index zero.  \end{proof}

\subsection{Completion of the proof}

Theorem~\ref{thm} now follows from Proposition~\ref{proposition: intermediate}, Lemma~\ref{lem: divergence of W}, and Lemma~\ref{lem: final}.

\begin{appendices} 

\smallskip

\section{Geometric preliminaries}\label{sec: appendix}

We collect a few notations and preliminaries on differential geometry in the appendix.

For an $n$-dimensional Riemannian manifold $(M,g)$,	we denote by $\sharp: TM \to T^*M$, $v \mapsto v^\sharp$ the canonical isomorphism between tangent and cotangent bundles --- that is, if $\{\p_1,\ldots, \p_n\} \subset \G(TM)$ is a local co-ordinate frame such that $g = g_{ij} \dd x^i \otimes \dd x^j$ with respect to its dual coframe $\{\dd x^1,\ldots,\dd x^n\} \subset \G(T^*M)$, then $v = v^i \p_i$ if and only if $v^\sharp = v_j \dd x^j$, where $v_j = g_{ij} v^i$. The inverse operator of $\sharp$ is denoted as $\flat$. 
	
Also, write $$\Omega^j(M) := \G\left( \bigwedge^j T^*M \right)$$ for the space of differential $j$-forms on $M$. For the exterior differential $\dd: \Omega^j(M) \to \Omega^{j+1}(M)$,  define as usual the codifferential $\dd^*: \Omega^{j+1}(M) \to \Omega^j(M)$ as the  $L^2$-adjoint of $\dd$, namely that with respect to the $L^2$-inner product $\bra\bullet,\bullet\ket$ induced by $g$. In this setting,
\begin{equation*}
{\rm div}\, v = \dd^*\left( v^\sharp\right)\qquad \text{ for } v \in \G(TM).
\end{equation*}
Moreover, the Laplace--Beltrami operator $\Delta_g$ defined on the space of differential $p$-forms  $\Omega^p(M)$ is defined by $$\Delta_g = \dd \dd^* + \dd^*\dd.$$ We may also write $\bra\bullet,\bullet\ket = \bra\bullet,\bullet\ket_g$ to stress that the inner product is taken with respect to $g$. 

For $j \in \{0,1,\ldots,n\}$, let $\star : \Omega^j(M) \to \Omega^{n-j}(M)$  be the Hodge star operator. We emphasise that $\star$ in our work is taken with respect to the Riemannian metric $g$. That is, for any $\omega \in \Omega^j(M)$, define $\star \omega$ as the $(n-j)$-form such that for any $\zeta \in \Omega^j(M)$, 
\begin{equation*}
\zeta \wedge \star \omega = \bra \zeta,\omega \ket\,\dvg.
\end{equation*}
Here $\dvg \in \Omega^n(M)$ is the Riemannian volume form with respect to $g$.

For a Riemannian manifold $(M,g)$, $\na$ always denotes the Levi-Civita connection, $\ball(x,r)$ denotes the geodesic balls, $d_g$ is the Riemannian distance, $|\bullet|$ is the Riemannian length of vectors, and $\volg$ is the Riemannian volume on $(M,g)$. In addition, we write ${\rm Riem}_g$, ${\rm Ric}_g$, and ${\rm sec}_g$ for the Riemann, Ricci, and sectional curvatures on $(M,g)$, respectively.

The above discussions on $\sharp$, $\flat$, $\dd$, $\dd^*$, $\star$... all carry over to geometric quantities with Sobolev regularity. The superscripts/subscripts indicating the dependence on $g$ will be suppressed when there is no danger of confusion for the metric.

\end{appendices}

\bigskip

\noindent
{\bf Acknowledgement}. SL is deeply indebted to Prof.~Alexander Shnirelman for very insightful discussions, when the author undertook a CRM--ISM postdoctoal fellowship at the Centre de Recherches Math\'{e}matiques and the Institut des Sciences Math\'{e}matiques in Montr\'{e}al during 2017--2019. SL also thanks McGill University and Concordia University  for providing pleasant working environments.

\end{document}